\documentclass{amsart}

\newtheorem{theorem}{Theorem}[section]

\newtheorem{lemma}[theorem]{Lemma}
\newtheorem{remark}[theorem]{Remark}

\title[on A GENERALIZED WIRTINGER INEQUALITY]
{on A GENERALIZED WIRTINGER INEQUALITY}
\author[Gisella Croce and Bernard Dacorogna]{}

\thanks{Research supported by Fonds National
Suisse (21-61390-00)} \email{gisella.croce@epfl.ch,
bernard.dacorogna@epfl.ch}


\begin{document}
\maketitle

\centerline{\scshape  Gisella Croce and Bernard Dacorogna}
\medskip

{\footnotesize \centerline{ D\'{e}partement de Math\'{e}matiques }
\centerline{ EPFL } \centerline{ 1015 Lausanne  CH } }
\medskip

\bigskip
\begin{quote}{\normalfont\fontsize{8}{10}\selectfont
{\bfseries Abstract.} Let
\[
\alpha\left(  p,q,r\right)  =\inf\left\{  \frac{\left\|
u^{\prime}\right\| _{p}}{\left\|  u\right\|  _{q}}:u\in
W_{per}^{1,p}\left(  -1,1\right)
\backslash\left\{  0\right\}  ,\;\int_{-1}^{1}\left|  u\right|  ^{r-2}%
u=0\right\}  .
\]
We show that%
\begin{align*}
\alpha\left(  p,q,r\right)   &  =\alpha\left(  p,q,q\right) \text{
if }q\leq
rp+r-1\\
\alpha\left(  p,q,r\right)   &  <\alpha\left(  p,q,q\right) \text{
if }q>\left(  2r-1\right)  p
\end{align*}
generalizing results of Dacorogna-Gangbo-Sub\'{\i}a and others.
\par}
\end{quote}

\vspace{0.3 cm}

\section{The main result}
In the present article we discuss the following minimization problem:%
\[
\alpha\left(  p,q,r\right)  =\inf\left\{  \frac{\left\|
u^{\prime}\right\| _{p}}{\left\|  u\right\|  _{q}}:u\in
W_{per}^{1,p}\left(  -1,1\right)
\backslash\left\{  0\right\}  ,\;\int_{-1}^{1}\left|  u\right|  ^{r-2}%
u=0\right\},
\]
where $p>1$,\ $q\geq r-1\geq1$ and%
\begin{align*}
\left\|  u\right\|  _{q} &  =\left(  \int_{-1}^{1}\left|  u\right|
^{q}\right)  ^{1/q},\\
W_{per}^{1,p}\left(  -1,1\right)   &  =\left\{  u:u\in
W^{1,p}\left( -1,1\right)  \text{ and }u\left(  -1\right) =u\left(
1\right)  \right\}  .
\end{align*}
We will denote by $p^{\prime}$\ the conjugate exponent of $p$\
(i.e. $\frac {1}{p}+\frac{1}{p^{\prime}}=1$) and the Beta function
\[
B\left(  p,q\right)  =\frac{\Gamma\left(  p\right)  \Gamma\left(
q\right) }{\Gamma\left(  p+q\right)  }=\int_{0}^{1}t^{p-1}\left(
1-t\right) ^{q-1}dt\,.
\]
Our main result will be

\begin{theorem}
\label{Theoreme principal}Let $p>1$,\ $q\geq r-1\geq1$; then%
\begin{align*}
\alpha\left(  p,q,r\right)   &  =\alpha\left(  p,q,q\right) \text{
if }q\leq
rp+r-1\\
\alpha\left(  p,q,r\right)   &  <\alpha\left(  p,q,q\right) \text{
if }q>\left(  2r-1\right)  p.
\end{align*}

Furthermore%
\[
\alpha\left(  p,q,q\right)  =2\left( \frac{1}{p^{\prime}}\right)
^{\frac {1}{q}}\left( \frac{1}{q}\right)
^{\frac{1}{p^{\prime}}}\left(  \frac {2}{p^{\prime}+q}\right)
^{\frac{1}{p}-\frac{1}{q}}B\left(  \frac
{1}{p^{\prime}},\frac{1}{q}\right)  .
\]
The above formula is also valid when $q=r>1$, $q=1$ ($p>1$ and
$r=2$) and $p=\infty$ ($q\geq r-1\geq1$).
\end{theorem}

\begin{remark}
(i) If the domain of integration is $\left(  a,b\right)  $\
instead of
$\left(  -1,1\right)  $\ the best constant becomes%
\[
\alpha_{a,b}\left(  p,q,r\right)  =\left(  \frac{2}{b-a}\right)
^{\frac {1}{p^{\prime}}+\frac{1}{q}}\alpha\left(  p,q,r\right)  .
\]

(ii) The case $p=q=2$\ is the classical Wirtinger inequality and
the constant
is then%
\[
\alpha\left(  2,2,2\right)  =\pi.
\]

(iii) The nonlinear case has first been investigated by
Dacorogna-Gangbo-Sub\'{\i}a \cite{Dacorogna-Gangbo-Subia} where
the cases $r=q$\ and $r=2$\ were considered. They computed the
actual value of $\alpha\left(  p,q,q\right)  $ and proved that
when $q\leq2p$\ then $\alpha\left(  p,q,2\right)  =\alpha\left(
p,q,q\right)  $, while for
$q>>2p$\ then strict inequality holds, showing in particular that%
\[
\alpha\left(  p,\infty,2\right)  =2^{1/p}\left(
p^{\prime}+1\right) ^{1/p^{\prime}}\,.
\]
The problem with $r=2$\ was then improved by many authors.
Belloni-Kawohl \cite{Belloni-Kawohl} and Kawohl \cite{Kawohl}
proved that the range where equality holds can be extended to
$2p+1$. Buslaev-Kondratiev-Nazarov\
\cite{Buslaev-Kondratiev-Nazarov}, refining a result from Egorov
\cite{Egorov}, showed that strict inequality holds as soon as
$q>3p$.

(iv) The importance of these best constants is, when $r=q$, to
generalize an isoperimetric inequality known as Wulff theorem, cf.
\cite{Dacorogna-Gangbo-Subia} (see also Lindquist-Peetre
\cite{Lindquist-Peetre}). The case $r=2$\ is important in many
different contexts, see for example \cite{Egorov-Kondratiev},
\cite{Lindquist}, \cite{Manasevich-Mawhin} or \cite{Otani}.

(v) In \cite{Dacorogna-Gangbo-Subia} the cases $r=q$\ and $r=2$
were treated separately. One of the aims of the present article
is, by the introduction of the parameter $r$, to unify these
treatments and, at the same time, to generalize the known results.

(vi) We would like to conclude this introduction by calling the
attention to some problems that we were not able to resolve. It is
believed, and supported by some numerical evidences, that the
equality between $\alpha\left( p,q,r\right)  $ and $\alpha\left(
p,q,q\right)  $ breaks down at exactly $\left(  2r-1\right)  p$
(B. Kawohl informed us that A.I. Nazarov \cite{NNuovo}, has
recently shown that when $r=2$\ the equality does indeed hold when
$q\leq3p$). A related question is to know the actual value of
$\alpha\left( p,q,r\right) $\ when the equality breaks down.
Another problem is to know for which $r\in\left[  2,q+1\right] $\
$\underset{2\leq r\leq q+1}{\min}\left\{ \alpha\left( p,q,r\right)
\right\}  $ is attained. By the theorem we know
that%
\[
\underset{2\leq r\leq q+1}{\max}\left\{  \alpha\left( p,q,r\right)
\right\} =\alpha\left(  p,q,q\right)  .
\]
\end{remark}

\section{Proof of the main result}

We proceed first with three lemmas and then with the proof of the
theorem.

\begin{lemma}
Let $p>1$,\ $q\geq r-1\geq1$. Let $F,K:\left(  0,1\right]
\rightarrow
\mathbb{R}$\ be defined by%

\[
K\left(  m\right)  =2\left(  \frac{p^{\prime}}{q}\right)  ^{\frac{1}%
{p^{\prime}}}\!\!\left[  \frac{q(p-1)+p}{2p\left(  1-r\left(
m\right)  \right) }\right]
^{\frac{p^{\prime}+q}{p^{\prime}q}}\!\!\int_{-m}^{1}\left[
1-r\left( m\right)  +r\left(  m\right)  \left|  z\right|
^{r-2}z-\left| z\right| ^{q}\right]  ^{\frac{1}{p^{\prime}}}\!dz
\]%
\[
F\left(  m\right)  =%
{\displaystyle\int_{-m}^{1}}
\frac{\left|  z\right|  ^{r-2}z}{\left[  1-r\left(  m\right)
+r\left( m\right)  \left|  z\right|  ^{r-2}z-\left|  z\right|
^{q}\right]  ^{1/p}}dz
\]
where
\[
r\left(  m\right)  =\frac{1-m^{q}}{1+m^{r-1}}.
\]
The following then holds%
\[
\alpha\left(  p,q,r\right)  =\inf\left\{  K\left(  m\right)
:m\in\left( 0,1\right]  \text{ and }F\left(  m\right) =0\right\} .
\]
\end{lemma}

\begin{proof}
The proof is similar in spirit to the one of
\cite{Dacorogna-Gangbo-Subia} but it differs in many technical
aspects.

\textit{Step 1 (Existence of minima).} The minimum is easily seen
to be
attained. Moreover there exists a minimum $u$ that satisfies also%
\[
u=u\left(  x\right)  =u\left(  p,q,r,x\right)  \in
W_{0}^{1,p}\left( -1,1\right)
\]
(which means, in particular, that we can assume that $u\left(
-1\right) =u\left(  1\right)  =0$) so that
\[
\alpha\left(  p,q,r\right)  =\frac{\left\|  u^{\prime}\right\|
_{p}}{\left\| u\right\|  _{q}}\text{ and}\;\int_{-1}^{1}\left|
u\right|  ^{r-2}u=0.
\]
In this step we only need $p,r>1$, $q\geq1$.

\textit{Step 2 (Euler-Lagrange equation)}. The function $u$ found
in the
preceding step satisfies%
\[
u,\;\left|  u^{\prime}\right|  ^{p-2}u^{\prime}\in C^{1}\left(
\left[ -1,1\right]  \right)
\]
and there exists $\mu\in\mathbb{R}$\ so that%
\begin{equation}
\label{Euler-Lagrange (lemme)} p\left(  \left| u^{\prime}\right|
^{p-2}u^{\prime}\right) ^{\prime
}+p\,\alpha^{p}\left\|  u\right\|  _{q}^{p-q}\left|  u\right|  ^{q-2}%
u-\mu(r-1)\left|  u\right|  ^{r-2}=0.
\end{equation}
Before briefly explaining how this equation can be derived, we
want to point out that it can be shown, with simple arguments,
that the Lagrange multiplier $\mu=0$ when $q=r$. It is this fact
that makes the whole analysis easier when $q=r$\ and that allows
also to treat the case $1<q=r<2$; however we do not discuss this
case in details and we refer to \cite{Dacorogna-Gangbo-Subia}.

\noindent Let $u$\ be a minimum and let $\varphi,\;\theta\in
C_{0}^{\infty }(-1,1)$ with
$(r-1)\int_{-1}^{1}{|u|^{r-2}\theta}=1$ and let $|\varepsilon
|,\;\left|  t\right|  <1$. Define then%
\[
\Phi(\varepsilon,t)=\int_{-1}^{1}{|u^{\prime}+\varepsilon\varphi^{\prime
}+t\theta^{\prime}|^{p}}-\alpha^{p}\left[
\int_{-1}^{1}{|u+\varepsilon \varphi+t\theta|^{q}}\right]
^{\frac{p}{q}},
\]%
\[
\Psi(\varepsilon,t)=\int_{-1}^{1}{|u+\varepsilon\varphi+t\theta|^{r-2}%
[u+\varepsilon\varphi+t\theta]}.
\]
It is easily seen that $\Phi \in C^1$ and $\Psi\in C^1$ if
$r\geq2$ and that $\Psi_{t}(0,0)=1\neq0$ for any choice of
$\theta$ as above. Therefore, applying the implicit function
theorem to $\Psi$, we find that there exist $\varepsilon_{0}<<1$
and a function $\tau\in C^{1}(-\varepsilon _{0},\varepsilon_{0})$,
with $\tau(0)=0$ such that $\Psi(\varepsilon
,\tau(\varepsilon))=0$,\ $\forall\varepsilon\in(-\varepsilon_{0}%
,\varepsilon_{0})$; in particular we deduce that
$\tau^{\prime}(0)=(1-r)\int _{-1}^{1}{|}u{|^{r-2}}{\varphi}$.
Since $\Phi(\varepsilon,\tau(\varepsilon))$ is minimum at
$\varepsilon=0$ we deduce that $\Phi_{\varepsilon}(0,0)+\Phi
_{t}(0,0)\tau^{\prime}(0)=0$. This leads to the Euler-Lagrange
equation in the
weak form which holds for every $\varphi\in C_{0}^{\infty}(-1,1)$, namely%
\[
p\int_{-1}^{1}|u^{\prime}|^{p-2}u^{\prime}\,\varphi^{\prime}-\alpha
^{p}\,p\,\Vert
u\Vert_{q}^{p-q}\int_{-1}^{1}|u|^{q-2}u\,\varphi+\mu
(r-1)\,\int_{-1}^{1}|u|^{r-2}\,\varphi=0
\]
where $\mu=\mu\left(  \alpha,\theta,u\right)  =-\Phi_{t}(0,0)\in\mathbb{R}%
$\ is a constant. We then deduce that $\left|  u^{\prime}\right|
^{p-2}u^{\prime}\in C^{1}$ and that (\ref{Euler-Lagrange (lemme)})
holds.
Moreover since the function $g\left(  t\right)  =\left|  t\right|  ^{p-2}%
t$\ has a continuous inverse, we have that
$u^{\prime}=g^{-1}\left( \left| u^{\prime}\right|
^{p-2}u^{\prime}\right)  $ is continuous and hence $u\in C^{1}$.

\noindent Note, for further reference, that we also have%
\begin{equation}
\left(  \left|  u^{\prime}\right|  ^{p}\right)
^{\prime}=p^{\prime}\left( \left|  u^{\prime}\right|
^{p-2}u^{\prime}\right)  ^{\prime}u^{\prime}\,.
\label{uprimepprime=...}%
\end{equation}
This is obviously true if $u\in C^{2}$. In our context this
follows from the
fact that the functions $f\left(  v\right)  =\left|  v\right|  ^{p^{\prime}}%
$\ and $v=\left|  u^{\prime}\right|  ^{p-2}u^{\prime}$ are both
$C^{1}$ which hence implies the claim, namely
\[
\left(  \left|  u^{\prime}\right|  ^{p}\right)  ^{\prime}=\left(
f\left( \left|  u^{\prime}\right|  ^{p-2}u^{\prime}\right) \right)
^{\prime }=f^{\prime}\left(  \left|  u^{\prime}\right|
^{p-2}u^{\prime}\right) \left(  \left|  u^{\prime}\right|
^{p-2}u^{\prime}\right)  ^{\prime }=p^{\prime}u^{\prime}\left(
\left|  u^{\prime}\right|  ^{p-2}u^{\prime }\right)  ^{\prime}\,.
\]

\textit{Step 3 (Integrated Euler-Lagrange equation)}. Multiplying
the preceding equation by $u^{\prime}$, using
(\ref{uprimepprime=...}) and
integrating we get ($c$ being a constant)%
\begin{equation}
\left(  p-1\right)  \left|  u^{\prime}\left(  x\right)  \right|
^{p}+\frac {p}{q}\alpha^{p}\left\|  u\right\|  _{q}^{p-q}\left|
u\left(  x\right) \right|  ^{q}-\mu\left|  u\left(  x\right)
\right|  ^{r-2}u\left(  x\right)
=c. \label{integratedee}%
\end{equation}
It is this version of the Euler-Lagrange equation that we will
almost always use (but not exclusively).

\noindent We propose, although this is not necessary for the
future developments, to derive directly (\ref{integratedee})
without using (\ref{Euler-Lagrange (lemme)}). The advantage of
this direct derivation is that it is also valid if $1<r<2$;
however it is not clear how to infer the required regularity of
$u$\ from (\ref{integratedee}). We now sketch the proof of this
fact. Consider the functional
\[
G(v)=\Vert v^{\prime}\Vert_{p}^{p}-\alpha^{p}\Vert v\Vert_{q}^{p}%
\]
where $v\in\mathcal{W}_{r}=\left\{  v\in W_{0}^{1,p}(-1,1)\text{ and }%
\int_{-1}^{1}\left|  v\right|  ^{r-2}v=0\right\}  $. We know that
it has a minimum at $u$. Consider for any $\varphi\in
W_{0}^{1,\infty}(-1,1)$\ such that
$\int_{-1}^{1}|u|^{r-2}u{\,\varphi}^{\prime}=0$ and for
$|\varepsilon
|\leq1$ the function%
\[
w_{\varepsilon}(x)=x+\varepsilon\frac{\varphi(x)}{2\left\|
\varphi^{\prime }\right\|  _{\infty}}.
\]
Observe that $w_{\varepsilon}:\left[  -1,1\right]
\rightarrow\left[ -1,1\right]  $ is a homeomorphism. It is easy to
see that if $v_{\varepsilon
}\left(  x\right)  =u(w_{\varepsilon}^{-1}(x))$, then $v_{\varepsilon}%
\in\mathcal{W}_{r}$ and therefore $G(u)=0\leq G(v_{\varepsilon})$,
which in
turn implies that%
\[
\left.  \frac{d}{d\varepsilon}G(v_{\varepsilon})\right|
_{\varepsilon=0}=0.
\]
We then deduce that
\begin{equation}
(1-p)\int_{-1}^{1}{\frac{|u^{\prime}(t)|^{p}\varphi^{\prime}(t)}{2\left\|
\varphi^{\prime}\right\|
_{\infty}}}dt=\alpha^{p}\,\frac{p}{q}{\left\|
u\right\|  }_{q}^{p-q}\int_{-1}^{1}\frac{|u(t)|^{q}\varphi^{\prime}%
(t)}{2\left\|  \varphi^{\prime}\right\|  _{\infty}}dt\,. \label{eqintegrata}%
\end{equation}
In order to have a more classical weak form of the integrated
Euler-Lagrange
equation, we need to remove the hypothesis that $\int_{-1}^{1}|u|^{r-2}%
u{\,\varphi}^{\prime}=0$. To do this we let $\psi\in
W_{0}^{1,\infty}(-1,1)$ be arbitrary and we choose
\[
\varphi\left(  x\right)  =\psi(x)-\left[  \int_{-1}^{1}|u(t)|^{r-2}%
u(t)\psi^{\prime}(t)dt\right]  f(x),
\]
where%
\[
f(x)=\frac{\int_{-1}^{x}|u(s)|^{r-2}u(s)ds}{\int_{-1}^{1}|u(a)|^{2r-2}da}.
\]
With this choice we obtain from (\ref{eqintegrata}) that%
\[
(1-p)\int_{-1}^{1}|u^{\prime}|^{p}\psi^{\prime}+\sigma\int_{-1}^{1}%
|u|^{r-2}u\psi^{\prime}-\alpha^{p}\,\frac{p}{q}{\left\|  u\right\|  }%
_{q}^{p-q}\int_{-1}^{1}|u|^{q}\psi^{\prime}=0
\]
for an appropriate $\sigma=\sigma\left(  \alpha,f,u\right)  \in\mathbb{R}%
$\ (it can be proved that it is identical to the $\mu$\ in
(\ref{integratedee})) and for any $\psi\in
W_{0}^{1,\infty}(-1,1)$. The integrated form (\ref{integratedee})
follows then immediately.

\textit{Step 4 (Value of }$\mu$\textit{, }$c$\textit{\ and
}$\left\| u\right\|  _{q}$\textit{ in terms of }$m$\textit{)}.
First observe that by
rescaling $u$, we can assume that%
\begin{align*}
\underset{x\in\left[  -1,1\right]  }{\max}\left\{  u\left(
x\right)
\right\}   &  =1\\
\underset{x\in\left[  -1,1\right]  }{\min}\left\{  u\left(
x\right) \right\}   &  =-m,\;m\in\left(  0,1\right]  .
\end{align*}
Writing
\begin{equation}
r\left(  m\right)  =\frac{1-m^{q}}{1+m^{r-1}}, \label{rdim}%
\end{equation}
we claim that
\begin{equation}
\mu=\frac{p}{q}\alpha^{p}\left\|  u\right\|  _{q}^{p-q}r\left(
m\right)  ,
\label{mu}%
\end{equation}%
\begin{equation}
c=\alpha^{p}\frac{p}{q}\left\|  u\right\|  _{q}^{p-q}(1-r(m)),
\label{costante}%
\end{equation}%
\begin{equation}
\left\|  u\right\|  _{q}=\left[ \frac{2p(1-r(m))}{q(p-1)+p}\right]
^{\frac{1}{q}}. \label{normau}%
\end{equation}
We start by establishing (\ref{mu}) and (\ref{costante}). Writing
the integrated Euler-Lagrange equation (\ref{integratedee}) for
the point of maximum $x_{0}$, i.e. $u(x_{0})=1$ (and
$u^{\prime}(x_{0})=0$), and for the point of minimum $x_{1}$, i.e.
$u(x_{1})=-m$ (and $u^{\prime}(x_{1})=0$), we get
\begin{equation}
\mu m^{r-1}+\alpha^{p}\frac{p}{q}\,\Vert
u\Vert_{q}^{p-q}m^{q}=-\mu+\alpha
^{p}\frac{p}{q}\,\Vert u\Vert_{q}^{p-q}=c. \label{valeurs de mu et c}%
\end{equation}
The identities (\ref{mu}) and (\ref{costante}) follow then
immediately. The
integrated Euler-Lagrange equation becomes then%
\begin{equation}
\left.
\begin{array}
[c]{c}%
(p-1)|u^{\prime}|^{p}+\alpha^{p}\frac{p}{q}\Vert u\Vert_{q}^{p-q}%
|u|^{q}-\alpha^{p}\frac{p}{q}\Vert u\Vert_{q}^{p-q}r(m)|u|^{r-2}u\\
=\alpha^{p}\frac{p}{q}\Vert u\Vert_{q}^{p-q}[1-r(m)].
\end{array}
\right.  \label{E-L integree avec mu et c}%
\end{equation}
or equivalently%
\begin{equation}
|u^{\prime}|=\left[  \frac{p^{\prime}}{q}\alpha^{p}\left\|
u\right\|
_{q}^{p-q}\right]  ^{\frac{1}{p}}[1-r(m)+r(m)|u|^{r-2}u-|u|^{q}]^{\frac{1}{p}%
}. \label{E-L integree avec mu et c (equivalent)}%
\end{equation}
Integrating (\ref{E-L integree avec mu et c}) over $\left(
-1,1\right)  $ and recalling that $\left\|  u^{\prime}\right\|
_{p}=\alpha\left\|  u\right\| _{q}$\ we get (\ref{normau}).

\textit{Step 5 (Qualitative properties of the solution)}. We now show that%
\begin{equation}
u^{\prime}\left(  x\right)  =0\Longleftrightarrow u\left( x\right)
=1\text{
or }u\left(  x\right)  =-m. \label{Derivee 0 en -m et 1}%
\end{equation}
Indeed the implication ($\Leftarrow$) is trivial. We next discuss
the counter
implication. Let%
\[
g\left(  X\right)  =1-r(m)+r(m)|X|^{r-2}X-|X|^{q},\;X\in\left[
-m,1\right]
\]
so that%
\[
|u^{\prime}|=\left[  \frac{p^{\prime}}{q}\alpha^{p}\left\|
u\right\| _{q}^{p-q}\right]  ^{\frac{1}{p}}[g\left(  u\right)
]^{\frac{1}{p}}.
\]
Observe that $g\left(  -m\right)  =g\left(  1\right)  =0$. Using
the hypothesis $q\geq r-1$\ one easily shows that at points
$\overline{X}$\ where $g^{\prime}\left(  \overline{X}\right) =0$\
then $g\left(  \overline {X}\right)  >0$. This shows that the
function $g$ never vanishes in $\left( -m,1\right)  $. This
implies that $u^{\prime}\left(  x\right)  \neq0$ if$\ u\left(
x\right)  \neq1$ and $u\left(  x\right)  \neq-m$, as claimed.

\noindent It can then be proved, exactly as in
\cite{Dacorogna-Gangbo-Subia} and we omit the details, that $u$\
has only one zero $\alpha\in\left(
-1,1\right)  $ ($\pm1$ being, by Step 1, also zeroes of $u$) and $u^{\prime}%
$\ has only two zeroes $\eta_{1}=\left(  \alpha-1\right)  /2$ and
$\eta _{2}=\left(  \alpha+1\right)  /2$. Furthermore the function
is symmetric in
the following sense%
\[
u\left(  x\right)  =\left\{
\begin{array}
[c]{cc}%
u\left(  2\eta_{1}-x\right)  & \text{if }x\in\left[  -1,\alpha\right] \\
\, & \\
u\left(  2\eta_{2}-x\right)  & \text{if }x\in\left[
\alpha,1\right]  .
\end{array}
\right.
\]
As a consequence we obtain that, for every continuous function $f:\mathbb{R}%
^{2}\rightarrow\mathbb{R}$,%
\begin{equation}
\int_{-1}^{1}f\left(  u\left(  x\right)  ,|u^{\prime}\left(
x\right) |\right)  dx=2\int_{\eta_{1}}^{\eta_{2}}f\left( u\left(
x\right)
,|u^{\prime}\left(  x\right)  |\right)  dx. \label{symetrie de u (lemme)}%
\end{equation}

\textit{Step 6 (The functions }$K$\textit{\ and }$F$\textit{)}.
Let $-m$\ be
the minimal value of the solution $u$. We will then establish that%
\begin{equation}
\left.
\begin{array}
[c]{c}%
\alpha=K\left(  m\right)  \equiv\\
2\!\left(  \frac{p^{\prime}}{q}\right)
^{\frac{1}{p^{\prime}}}\left[ \frac{q(p-1)+p}{2p\left( 1-r\left(
m\right)  \right)  }\right]
^{\frac{p^{\prime}+q}{p^{\prime}q}}%
{\displaystyle\int_{-m}^{1}}
\left[  1-r\left(  m\right)  +r\left(  m\right)  \left|  z\right|
^{r-2}z-\left|  z\right|  ^{q}\right]  ^{\frac{1}{p^{\prime}}}\!dz
\end{array}
\right.  \label{alpha}%
\end{equation}%
\begin{equation}
F\left(  m\right)  \equiv%
{\displaystyle\int_{-m}^{1}}
\frac{\left|  z\right|  ^{r-2}z}{\left[  1-r\left(  m\right)
+r\left(
m\right)  \left|  z\right|  ^{r-2}z-\left|  z\right|  ^{q}\right]  ^{1/p}%
}dz=0. \label{fdiemme}%
\end{equation}
We now briefly explain how to derive these identities. We start
with (\ref{alpha}). (Note that the derivation here is done in a
slightly different manner than in \cite{Dacorogna-Gangbo-Subia}.
There a function $L$ was derived instead of the present function
$K$ below; they coincide at the minimal value). Using (\ref{E-L
integree avec mu et c}) we obtain
\[
|u^{\prime}|^{p}=\left[  \frac{p^{\prime}}{q}\alpha^{p}\left\|
u\right\|
_{q}^{p-q}\right]  ^{\frac{1}{p^{\prime}}}[1-r(m)+r(m)|u|^{r-2}u-|u|^{q}%
]^{\frac{1}{p^{\prime}}}|u^{\prime}|.
\]
Let $\eta_{1},\eta_{2}$ be the zeroes of $u^{\prime}$. Recalling
(\ref{symetrie de u (lemme)}), integrating the above equation over
$\left( \eta_{1},\eta_{2}\right)  $ and performing the change of
variable $z=u\left(
x\right)  $\ in the right hand side of the equation, we obtain%
\[
\left\|  u^{\prime}\right\|  _{p}^{p}=2\left(
\frac{p^{\prime}}{q}\alpha ^{p}\left\|  u\right\|
_{q}^{p-q}\right)  ^{\frac{1}{p^{\prime}}}\int
_{-m}^{1}[1-r(m)+r(m)|z|^{r-2}z-|z|^{q}]^{\frac{1}{p^{\prime}}}dz
\]
which combined with (\ref{normau}) and with $\left\|
u^{\prime}\right\| _{p}=\alpha\left\|  u\right\|  _{q}$ implies
(\ref{alpha}).

\noindent To obtain (\ref{fdiemme}) we rewrite the condition $\int_{-1}%
^{1}\left|  u\right|  ^{r-2}u=0$ in the following manner. We first
observe that ($\eta_{1},\eta_{2}$ being the zeroes of
$u^{\prime}$), appealing to (\ref{E-L integree avec mu et c
(equivalent)}) and (\ref{symetrie de u (lemme)}), we have
\begin{align*}
0  &  =\int_{-1}^{1}|u|^{r-2}u=2\int_{\eta_{1}}^{\eta_{2}}|u|^{r-2}%
u=2\int_{\eta_{1}}^{\eta_{2}}\frac{|u|^{r-2}uu^{\prime}}{u^{\prime}}\\
&  =2\left[  \frac{p^{\prime}}{q}\alpha^{p}\left\|  u\right\|  _{q}%
^{p-q}\right]  ^{\frac{-1}{p}}%
{\displaystyle\int_{\eta_{1}}^{\eta_{2}}}
\frac{|u|^{r-2}uu^{\prime}}{\left[  1-r\left(  m\right) +r\left(
m\right) \left|  u\right|  ^{r-2}u-\left|  u\right| ^{q}\right]
^{1/p}}.
\end{align*}
Performing the change of variable $z=u\left(  x\right)  $ we get
(\ref{fdiemme}).

\textit{Step 7 (Equivalence of minima)}. Denote by%
\[
\beta=\inf\left\{  K\left(  m\right)  :m\in\left(  0,1\right]
\text{ and }F\left(  m\right)  =0\right\}  .
\]
The aim of this step is to show that $\alpha=\beta$ concluding
thus the proof of the lemma. From the previous steps we know that
$\beta\leq\alpha$. We now wish to show the reverse inequality. Let
$m\in(0,1]$ be such that $\beta=K\left(  m\right)  $\ and $F\left(
m\right)  =0$\ (such an $m$ exists by continuity of the functions
$F$ and $K$\ and by the fact that $F\left( 0\right) \neq0$). To
conclude to the inequality $\alpha\leq\beta$\ it is
enough to show that we can find $u\in W_{per}^{1,p}(-1,1)$ with $\int_{-1}%
^{1}\left|  u\right|  ^{r-2}u=0$ such that
\begin{equation}
K(m)=\frac{\left\|  u^{\prime}\right\|  _{p}}{\left\|  u\right\|
_{q}}\ \,\, \textnormal{and}\,\, m=-\min\limits_{x\in [-1,1]}
u(x)\,.
\label{K(m)=alpha}%
\end{equation}
This $u$\ will be constructed as follows. We claim that we can
find $u\in W^{1,p}(-1,0)$ a solution of the problem
\[
(E_{m})\left\{
\begin{array}
[c]{l}%
u^{\prime}=\gamma h(u),\;x \in [-1,0) \\
u(-1)=-m,\;u(0)=1\\
\max u(x)=\max|u(x)|=1
\end{array}
\right.
\]
where
\[
h(s)=[1-r(m)+r\left(  m\right)  |s|^{r-2}s-|s|^{q}]^{\frac{1}{p}}%
\;,\;\gamma=\int_{-m}^{1}\frac{ds}{h(s)}.
\]
Note, for further reference, that since $h\left(  -m\right)
=h\left(
1\right)  =0$ \ then%
\[
u^{\prime}\left(  -1\right)  =u^{\prime}\left(  0\right)  =0.
\]
A solution of $(E_{m})$ is constructed as follows. Let
$H:[-m,1]\rightarrow \left[  -\gamma,0\right]  $ be defined by
$H(y)=\int_{1}^{y}\frac{dx}{h(x)}$.
The solution of $(E_{m})$\ is then given by%
\[
u\left(  x\right)  =H^{-1}(\gamma x).
\]
Using the fact that $F\left(  m\right)  =0$\ we obtain that%
\begin{align*}
\int_{-1}^{0}|u|^{r-2}u  &  =\int_{-1}^{0}\frac{|u|^{r-2}uu^{\prime}%
}{u^{\prime}}=\frac{1}{\gamma}\int_{-1}^{0}\frac{|u|^{r-2}uu^{\prime}%
}{h\left(  u\right)  }\\
&  =\frac{1}{\gamma}\int_{-m}^{1}\frac{\left|  z\right|
^{r-2}z}{h\left( z\right)  }dz=\frac{1}{\gamma}F\left(  m\right)
=0.
\end{align*}
We then extend $u$ to $(0,1]$\ so as to be even. It is then clear
that $u\in W_{per}^{1,p}(-1,1)$ and that
$\int_{-1}^{1}|u|^{r-2}u=0$.

\noindent It therefore remains to prove (\ref{K(m)=alpha}). From
$(E_{m})$, the fact that $\int_{-1}^{1}|u|^{r-2}u=0$ and the
evenness of $u$\ we deduce that
\begin{equation}
\Vert u^{\prime}\Vert_{p}^{p}=\gamma^{p}[2(1-r(m))-\Vert
u\Vert_{q}^{q}].
\label{(1) dans etape 6}%
\end{equation}
In a similar way we have from $(E_{m})$ that%
\[
\left(  u^{\prime}\right)  ^{p}=\gamma^{p-1}\left[  h(u)\right]
^{p-1}u^{\prime}\,;
\]
using the evenness of $u$, and after a change of variables\ we deduce that%
\[
\Vert
u^{\prime}\Vert_{p}^{p}=2\gamma^{p-1}\int_{-m}^{1}[h(s)]^{p-1}ds\,.
\]
Recalling the definition of $K\left(  m\right)  $ we obtain%
\begin{equation}
\Vert u^{\prime}\Vert_{p}^{p}=\gamma^{p-1}\!\left(  \frac{q}{p^{\prime}%
}\right)  ^{\frac{1}{p^{\prime}}}\left[  \frac{2p\left( 1-r\left(
m\right)
\right)  }{q\left(  p-1\right)  +p}\right]  ^{\frac{p^{\prime}+q}{p^{\prime}%
q}}K\left(  m\right)  \,. \label{(2) dans etape 6}%
\end{equation}
From $(E_{m})$ we also have%
\[
\left|  u^{\prime}\right|  ^{p}=\gamma^{p}[1-r(m)+r\left( m\right)
|u|^{r-2}u-|u|^{q}].
\]
Differentiating this equation, using (\ref{uprimepprime=...}), we
get after a
simplification by $u^{\prime}$\ that%
\[
p^{\prime}\left(  \left|  u^{\prime}\right|
^{p-2}u^{\prime}\right)
^{\prime}=\gamma^{p}\left[  \left(  r-1\right)  r(m)|u|^{r-2}-q|u|^{q-2}%
u\right]  .
\]
Multiplying this equation by $u$,\ integrating, bearing in mind
that $\int_{-1}^{1}|u|^{r-2}u=0$, that $u$\ is even and that
$u^{\prime}\left( -1\right)  =u^{\prime}\left(  1\right)  =0$, and
using (\ref{(1) dans etape
6}) we get%
\begin{equation}
\Vert
u^{\prime}\Vert_{p}^{p}=\frac{q}{p^{\prime}}\gamma^{p}\,\Vert
u\Vert _{q}^{q}\text{ and }\Vert
u\Vert_{q}^{q}=\frac{2p^{\prime}\left(  1-r\left( m\right) \right)
}{q+p^{\prime}}=\frac{2p\left(  1-r\left(  m\right)
\right)  }{q\left(  p-1\right)  +p}. \label{(3) dans etape 6}%
\end{equation}
Combining (\ref{(2) dans etape 6}) and (\ref{(3) dans etape 6}) we
find the
claimed result%
\[
K(m)=\frac{\left\|  u^{\prime}\right\|  _{p}}{\left\|  u\right\|
_{q}}.
\]

\end{proof}
\begin{remark}\label{referee}
We observed that $q=r$ implies $\mu=0$, and so $m=1$. From the
previous step, we know that $\alpha=K(m)$, so
$\alpha(p,q,q)=K(1).$
\end{remark}
 We now study the functions $F$ (cf. Lemma \ref{Etude de
F}) and $K$ (cf. Lemma \ref{Etude de K}).

\begin{lemma}
\label{Etude de F}Let $F:\left(  0,1\right] \rightarrow\mathbb{R}$
be the function defined in the preceding lemma. The following
properties then hold.

(i) $F\left(  1\right)  =0$, for every $p>1$\ and $q\geq
r-1\geq1$.

(ii) If $q\leq rp+r-1$\ then $F\left(  m\right)  \neq0$ for every
$m\in(0,1)$.

(iii) If $q>\left(  2r-1\right)  p$, then there exists $m\in\left(
0,1\right)  $ (i.e. $m\neq1$) such that $F\left(  m\right)  =0$.
Moreover $F<0$\ for $m$\ close to $1$ ($m<1$).
\end{lemma}

\begin{proof}
The function $F\in C^{1}\left(  \left(  0,1\right]  \right)  $\
and\ we can
rewrite it in the following way%
\[
F(m)=\int_{0}^{1}g_{m}(t)\,dt,
\]
where
\[
g_{m}(t)=\frac{t^{r-1}}{[1-r(m)+r(m)t^{r-1}-t^{q}]^{\frac{1}{p}}}%
-\frac{t^{r-1}m^{r}}{[1-r(m)-r(m)t^{r-1}m^{r-1}-m^{q}t^{q}]^{\frac{1}{p}}}.
\]

\textit{Step 1.} Note that since $r(m)=0$, recalling that $r\left(
m\right) =\frac{1-m^{q}}{1+m^{r-1}}$, whenever $m=1$, we deduce
that $g_{1}(t)\equiv0$ and thus $F(1)=0$.

\textit{Step 2.} We will now prove that, when $q\leq rp+r-1$,\
then
\[
g_{m}\left(  t\right)  \geq0,\;\forall t\in\left[  0,1\right]
\]
leading to the claim. Observe that $g_{m}(t)\geq0$ if and only if
\begin{equation}
\left.
\begin{array}
[c]{c}%
h_{m}(t)\equiv1-r(m)-r(m)t^{r-1}m^{r-1}-m^{q}t^{q}-m^{rp}[1-r(m)+r(m)t^{r-1}%
-t^{q}]\\
=\left(  1-m^{rp}\right)  \left(  1-r(m)\right)  -\left(  m^{r-1}%
+m^{rp}\right)  r\left(  m\right)  t^{r-1}-\left(
m^{q}-m^{rp}\right) t^{q}\geq0\,.
\end{array}
\right.  \label{hm positif}%
\end{equation}
Note that%
\begin{equation}
h_{m}^{\prime}(t)=-\left(  r-1\right)  \left(
m^{r-1}+m^{rp}\right)  r\left(
m\right)  t^{r-2}-q\left(  m^{q}-m^{rp}\right)  t^{q-1}. \label{h'm positif}%
\end{equation}
To establish (\ref{hm positif}) we divide the proof into two
cases.

Case 1: $q\leq rp$. Observe that, in this case, since $0<m\leq1$,
then
trivially $h_{m}^{\prime}(t)\leq0$. On the other hand $h_{m}(0)\geq0=h_{m}%
(1)$, therefore (\ref{hm positif}) is proved.

Case 2: $rp<q\leq rp+r-1$. We will show that if there exists $\overline{t}%
\in\left[  0,1\right]  $ with $h_{m}^{\prime}(\overline{t})=0$
then necessarily $h_{m}(\overline{t})\geq0$. This fact coupled
with the observation that $h_{m}(0)\geq0=h_{m}(1)$ shows (\ref{hm
positif}). Note that
$h_{m}^{\prime}(\overline{t})=0$ if and only if%
\[
\overline{t}^{q-r+1}=\frac{r(m)(r-1)(m^{rp}+m^{r-1})}{q\left(  m^{rp}%
-m^{q}\right)  }.
\]
We therefore have (assuming that $\overline{t}\leq1$,\ otherwise
nothing is to
be proved)%
\[%
\begin{array}
[c]{l}%
h_{m}(\overline{t})=[1-r(m)](1-m^{rp})-\overline{t}^{r-1}r(m)(m^{r-1}%
+m^{rp})\frac{q-r+1}{q}\\
\geq\lbrack1-r(m)](1-m^{rp})-r(m)(m^{r-1}+m^{rp})\frac{q-r+1}{q}\\
\geq\frac{1}{q}\,\frac{m^{r-1}}{1+m^{r-1}}\{q(m^{q-r+1}+1)(1-m^{rp}%
)-(q-r+1)(1-m^{q})(1+m^{rp+1-r})\}.
\end{array}
\]
To obtain the claim it is thus sufficient to show that, for every
$m\in\left[
0,1\right]  $,%
\[
G(m)\equiv
q(m^{q-r+1}+1)(1-m^{rp})-(q-r+1)(1-m^{q})(1+m^{rp+1-r})\geq0.
\]
$\,$Observe first that $G(0)\geq G(1)=0$. Define, for
$\alpha\geq0$,
\[
H(\alpha,m)=q(m^{\alpha-r+1}+1)(1-m^{rp})-(q-r+1)(1-m^{\alpha})(1+m^{rp+1-r}%
).
\]
Note that $H(q,m)=G(m)$. Moreover if $\alpha\geq\beta\geq0$, then
$H(\beta,m)\geq H(\alpha,m)$. If we can show that
$H(rp+r-1,m)\geq0$ we would obtain
\[
G(m)=H(q,m)\geq H(rp+r-1,m)\geq0
\]
as claimed. It therefore remains to show that, for every
$m\in\left[
0,1\right]  $,%
\begin{align*}
\widetilde{H}\left(  m\right)   &  \equiv H(rp+r-1,m)\\
&  =\left(  r-1\right)
(1-m^{2rp})+(q-r+1)(m^{rp+r-1}-m^{rp+1-r})\geq0.
\end{align*}
This is proved by observing that $\widetilde{H}\left(  0\right)
=r-1\geq\widetilde{H}\left(  1\right)  =0$ and that
$\widetilde{H}^{\prime
}\left(  m\right)  \leq0$. To prove this last inequality we first observe that%
\[
\widetilde{H}^{\prime}\left(  m\right) =-m^{rp+r-2}\varphi\left(
m\right)
\]
where%
\[
\varphi\left(  m\right)  =2rp\left(r\!-\!1\right)
m^{rp-r\!+\!1}+\left( q\!-\!r\!+\!1\right)  \left(
rp-\!r\!+\!1\right) m^{-2r+2}-\left( q\!-\!r\!+\!1\right)  \left(
rp+\!r\!-\!1\right)  .
\]
We next see that $\lim_{m \to 0}\varphi\left( m\right)= + \infty$
and from the hypothesis
of Case 2%
\[
\varphi\left(  1\right)  =2\left(  r-1\right)  \left(
rp+r-1-q\right) \geq0.
\]
To conclude to $\varphi\left(  m\right)  \geq0$\ for every
$m\in\left[ 0,1\right]  $, we observe that at a point
$\overline{m}$ where $\varphi ^{\prime}\left( \overline{m}\right)
=0$, one has $\varphi\left( \overline{m}\right)  \geq0$\ and this
concludes the proof.

\begin{remark}
The previous study of the inequality $G(m)\geq 0$ in $[0,1]$ is
optimal: in fact, we noticed that $G(0)\geq G(1)=0$; therefore,
$G(m)\geq 0$ implies necessarily $G'(1)\leq 0$. After a simple
computation we have $G'(1)\leq 0$ if and only if $q \leq rp+r-1,$
the same estimation we found in the proof.
\end{remark}

\textit{Step 3.} We will now prove that $F^{\prime}(1)>0$ if and
only if $q>(2r-1)p$. Since $F(1)=0$ and $F(0)>0$ this will show,
as wished, that there exists $m_{0}\in(0,1)$ such that
$F(m_{0})=0$ and that $F<0$\ for $m$\ close
to $1$ ($m<1$). A direct computation shows that%
\[
F^{\prime}(1)=\frac{q}{p}\int_{0}^{1}\frac{t^{2r-2}-t^{r-1+q}}{(1-t^{q}%
)^{1+\frac{1}{p}}}dt-r\int_{0}^{1}\frac{t^{r-1}}{(1-t^{q})^{\frac{1}{p}}%
}dt\,.
\]
Changing the variable $s=t^{q}$\ we get%
\begin{align*}
F^{\prime}(1)  &  =\frac{1}{p}\int_{0}^{1}\frac{s^{\frac{2r-1}{q}-1}%
-s^{\frac{r}{q}}}{(1-s)^{1+\frac{1}{p}}}ds-\frac{r}{q}\int_{0}^{1}%
\frac{s^{\frac{r}{q}-1}}{(1-s)^{\frac{1}{p}}}ds\\
&  =\frac{1}{p}\int_{0}^{1}\frac{s^{\frac{2r-1}{q}-1}-s^{\frac{2r-1}{q}}%
}{(1-s)^{1+\frac{1}{p}}}ds+\frac{1}{p}\int_{0}^{1}\frac{s^{\frac{2r-1}{q}%
}-s^{\frac{r}{q}}}{(1-s)^{1+\frac{1}{p}}}ds-\frac{r}{q}\int_{0}^{1}%
\frac{s^{\frac{r}{q}-1}}{(1-s)^{\frac{1}{p}}}ds\;.
\end{align*}
Note that the first expression is readily given as%
\[
\frac{1}{p}\int_{0}^{1}\frac{s^{\frac{2r-1}{q}-1}-s^{\frac{2r-1}{q}}%
}{(1-s)^{1+\frac{1}{p}}}ds=\frac{1}{p}\int_{0}^{1}\frac{s^{\frac{2r-1}{q}-1}%
}{(1-s)^{\frac{1}{p}}}ds=\frac{1}{p}B\left(  \frac{2r-1}{q},\frac{1}%
{p^{\prime}}\right)  .
\]
Integrating by parts the second term in $F^{\prime}\left( 1\right)
$\ and
applying L'H\^{o}pital's rule we obtain%
\begin{align*}
\frac{1}{p}\int_{0}^{1}\frac{s^{\frac{2r-1}{q}}-s^{\frac{r}{q}}}%
{(1-s)^{1+\frac{1}{p}}}ds  &  =\left.  \left[  \frac{s^{\frac{2r-1}{q}%
}-s^{\frac{r}{q}}}{(1-s)^{\frac{1}{p}}}\right]  \right|  _{0}^{1}-\int_{0}%
^{1}\frac{\frac{2r-1}{q}s^{\frac{2r-1}{q}-1}-\frac{r}{q}s^{\frac{r}{q}-1}%
}{(1-s)^{\frac{1}{p}}}ds\\
&
=-\frac{2r-1}{q}\int_{0}^{1}\frac{s^{\frac{2r-1}{q}-1}}{(1-s)^{\frac{1}{p}%
}}ds+\frac{r}{q}\int_{0}^{1}\frac{s^{\frac{r}{q}-1}}{(1-s)^{\frac{1}{p}}}ds\\
&  =-\frac{2r-1}{q}B\left(
\frac{2r-1}{q},\frac{1}{p^{\prime}}\right)
+\frac{r}{q}\int_{0}^{1}\frac{s^{\frac{r}{q}-1}}{(1-s)^{\frac{1}{p}}}ds\;.
\end{align*}
Combining these results we have%
\[
F^{\prime}\left(  1\right)  =\left(
\frac{1}{p}-\frac{2r-1}{q}\right) B\left(
\frac{2r-1}{q},\frac{1}{p^{\prime}}\right)
\]
which leads to the assertion.
\end{proof}

\begin{lemma}
\label{Etude de K}Let $q>\left(  2r-1\right)  p$, then there exists $m_{0}%
\in\left(  0,1\right)  $ (i.e. $m_{0}<1$) such that $\inf\left\{
K\left( m\right)  :m\in\left(  0,1\right]  \text{ and }F\left(
m\right)  =0\right\} =K\left(  m_{0}\right)  <K\left(  1\right) $.
\end{lemma}

\begin{proof}
\textit{Step 1.} As already mentioned it is easy to see that the
minimum is attained and we therefore wish to show that $m_{0}<1$
and $K\left( m_{0}\right)  <K\left(  1\right)  $. To this aim we
first observe that $K\in C^{1}\left(  \left(  0,1\right] \right)
$\ and that $\lim \limits_{m\rightarrow0}K(m)=+\infty$. We will
then prove, in the next step,
that there exists a constant $c(p,q)>0$ such that%
\[
K^{\prime}(m)=\frac{c\left(  p,q\right)
}{p^{\prime}}r^{\prime}\left(
m\right)  \left(  1-r\left(  m\right)  \right)  ^{\frac{1}{p}-\frac{1}{q}%
-2}\left(  1-\frac{r-1}{q}r\left(  m\right)  \right)  F\left(
m\right)  \,.
\]
Recall that $r^{\prime}(m)\,<0$\ for every $m\in\left( 0,1\right]
$. Since $F<0$\ for $m$\ close to $1$ ($m<1$) (by Lemma \ref{Etude
de F}), we deduce that $m=1$\ is a local maximum of $K$ in $\left(
0,1\right]  $. Therefore the global minimum of $K$ in $\left(
0,1\right]  $\ is at a point $m_{0}\in(0,1)$ where $F(m_{0})=0$.
This is the claimed result.

\textit{Step 2.} We now compute $K^{\prime}(m)$. Recall that%
\begin{align*}
K\left(  m\right)   &  =c\left(  p,q\right)  \left(  1-r\left(
m\right) \right)
^{-\frac{1}{p^{\prime}}-\frac{1}{q}}\int_{-m}^{1}\left[  1-r\left(
m\right)  +r\left(  m\right)  \left|  z\right|  ^{r-2}z-\left|
z\right|
^{q}\right]  ^{\frac{1}{p^{\prime}}}dz\\
&  =c\left(  p,q\right)  \left(  1-r\left(  m\right)  \right)
^{-\frac{1}{q}%
}\\
&  \int_{-m}^{1}\left[  1+r\left(  m\right)  \left(  1-r\left(
m\right) \right)  ^{-1}\left|  z\right|  ^{r-2}z-\left( 1-r\left(
m\right)  \right) ^{-1}\left|  z\right|  ^{q}\right]
^{\frac{1}{p^{\prime}}}dz
\end{align*}
where%
\[
c\left(  p,q\right)  =2\left(  \frac{p^{\prime}}{q}\right) ^{\frac
{1}{p^{\prime}}}\left[  \frac{q(p-1)+p}{2p}\right]  ^{\frac{p^{\prime}%
+q}{p^{\prime}q}}\,.
\]
Writing $z=\left(  1-r\left(  m\right)  \right)  ^{\frac{1}{q}}t$,
$\alpha\left(  m\right)  =-m\left(  1-r\left(  m\right)  \right)
^{-\frac {1}{q}}$, $\beta\left(  m\right)  =\left(  1-r\left(
m\right)  \right) ^{-\frac{1}{q}}$\ and $\gamma\left(  m\right)
=r\left(  m\right)  \left(
1-r\left(  m\right)  \right)  ^{-1+\frac{r-1}{q}}$\ we obtain%
\[
K\left(  m\right)  =c\left(  p,q\right)  \int_{\alpha\left(
m\right) }^{\beta\left(  m\right)  }\left[  1+\gamma\left(
m\right)  \left|  t\right| ^{r-2}t-\left|  t\right|  ^{q}\right]
^{\frac{1}{p^{\prime}}}dt\,.
\]
Noting that%
\[
1+\gamma\left(  m\right)  \left|  \beta\left(  m\right)  \right|
^{r-1}-\left|  \beta\left(  m\right)  \right| ^{q}=1-\gamma\left(
m\right) \left|  \alpha\left(  m\right) \right|  ^{r-1}-\left|
\alpha\left( m\right)  \right|  ^{q}=0
\]
we obtain%
\[
K^{\prime}\left(  m\right)  =\frac{c\left(  p,q\right)  }{p^{\prime}}%
\gamma^{\prime}\left(  m\right)
{\displaystyle\int_{\alpha\left(  m\right)  }^{\beta\left(
m\right)  }}
\frac{\left|  t\right|  ^{r-2}t}{\left[  1+\gamma\left( m\right)
\left| t\right|  ^{r-2}t-\left|  t\right|  ^{q}\right]
^{\frac{1}{p}}}dt\,.
\]
Performing backward the change of variable $t=\left(  1-r\left(
m\right)
\right)  ^{-\frac{1}{q}}z$\ we get%
\begin{align*}
K^{\prime}\left(  m\right)   &  =\frac{c\left(  p,q\right)  }{p^{\prime}%
}\gamma^{\prime}\left(  m\right)  \left(  1-r\left(  m\right)
\right)
^{-\frac{r}{q}}\\
&  \int_{-m}^{1}\frac{\left|  z\right|  ^{r-2}z\,dz}{\left[
1+r\left( m\right)  \left(  1-r\left(  m\right)  \right)
^{-1}\left|  z\right| ^{r-2}z-\left(  1-r\left(  m\right) \right)
^{-1}\left|  z\right|
^{q}\right]  ^{\frac{1}{p}}}\\
&  =\frac{c\left(  p,q\right) }{p^{\prime}}\gamma^{\prime}\left(
m\right) \left(  1-r\left( m\right)  \right)
^{\frac{1}{p}-\frac{r}{q}}F\left(
m\right) \\
&  =\frac{c\left(  p,q\right)  }{p^{\prime}}r^{\prime}\left(
m\right) \left(  1-r\left(  m\right)  \right)
^{\frac{1}{p}-\frac{1}{q}-2}\left( 1-\frac{r-1}{q}r\left( m\right)
\right)  F\left(  m\right)
\end{align*}
as wished.
\end{proof}

\vspace{0.2cm}

We are now in a position to conclude the proof of the main
theorem.

\vspace{0.2cm} \noindent
\begin{proof}
(Theorem \ref{Theoreme principal}). \textit{Step 1}. If $q\leq
rp+r-1$, then, since $F(m)=0$ if and only if $m=1$, we deduce
(recalling \ref{referee}) that
\[
\alpha(p,q,r)=\alpha(p,q,q)=K(1)=2\left(
\frac{p^{\prime}}{q}\right) ^{\frac{1}{p^{\prime}}}\left[
\frac{q(p-1)+p}{2p}\right]  ^{\frac{p^{\prime
}+q}{p^{\prime}q}}\int_{-1}^{1}\left[  1-\left|  z\right|
^{q}\right] ^{\frac{1}{p^{\prime}}}dz
\]
which easily leads to the value given in the theorem.

\noindent If $q>(2r-1)p$, we have, as claimed, that%
\[
\alpha(p,q,q)=K(1)>\inf\{K(m):\;m\in(0,1],\;F(m)=0\}=\alpha(p,q,r).
\]
It remains to discuss the limit cases.

\textit{{Step 2. }}The case $q=r\geq2$\ is part of the previous
analysis. If, however $1<q=r<2$, then the result still holds and
we refer to \cite{Dacorogna-Gangbo-Subia} for more details.

\textit{{Step 3. }}We now discuss the value of $\alpha(p,1,2)$. Let%
\[
\mathcal{X}_{2}=\left\{  v\in
W_{per}^{1,p}(-1,1),\;\int_{-1}^{1}v=0\right\} .
\]
We have just seen that for every $q>1$ then, using also H\"{o}lder
inequality,%
\[
\left\|  v^{\prime}\right\|  _{p}\geq\alpha(p,q,2)\left\|  v\right\|  _{q}%
\geq\alpha(p,q,2)2^{-\frac{1}{q^{\prime}}}\left\|  v\right\|
_{1}\;,\;\forall
v\in\mathcal{X}_{2}%
\]
and hence denoting by%
\[
\bar{\alpha}=\lim\limits_{q\rightarrow1}\alpha(p,q,2)=2^{\frac{1}{p}}%
\frac{{(p^{\prime}+1)}^{\frac{1}{p^{\prime}}}}{p^{\prime}}B\left(
\frac
{1}{p^{\prime}},1\right)  =2^{\frac{1}{p}}{(p^{\prime}+1)}^{\frac{1}%
{p^{\prime}}}%
\]
we get%
\[
\left\|  v^{\prime}\right\|  _{p}\geq\bar{\alpha}\left\| v\right\|
_{1}\;,\;\forall v\in\mathcal{X}_{2}.
\]
We therefore have just obtained that
$\alpha(p,1,2)\geq\overline{\alpha}$. We now prove the reverse
inequality. Let for $q>1$, $u_{q}\in\mathcal{X}_{2}$ be a
minimizer, i.e.
\[
\frac{\left\|  u_{q}^{\prime}\right\|  _{p}}{\left\|  u_{q}\right\|  _{q}%
}=\alpha(p,q,2)\,.
\]
Recall that with our conventions $-1\leq-m\leq u_{q}\left(
x\right)  \leq1$ and hence $\left\|  u_{q}\right\|
_{\infty}\leq1$. From the integrated Euler-Lagrange equation
(\ref{E-L integree avec mu et c (equivalent)}) we then get
$\left\|  u_{q}^{\prime}\right\|  _{\infty}\leq C(p)$. Therefore
there exists $\bar{u}$ and a subsequence, still denoted by
$u_{q}$, such that $u_{q}\overset{\ast}{\rightharpoonup}\bar{u}$
in $W^{1,\infty}$ and $u_{q}\rightarrow\bar{u}$ in $L^{\infty}$.
This implies by weak lower
semicontinuity%
\[
\left\|  \overline{u}^{\prime}\right\|  _{p}\leq\underset{q\rightarrow1}%
{\lim\inf}\left\|  u_{q}^{\prime}\right\|  _{p}\,.
\]
Moreover%
\begin{align*}
\left|  \left\|  u_{q}\right\|  _{q}-\left\|  \overline{u}\right\|
_{1}\right|   &  \leq\left\|  u_{q}-\overline{u}\right\|
_{q}+\left| \left\|  \overline{u}\right\|  _{q}-\left\|
\overline{u}\right\|  _{1}\right|
\\
&  \leq2^{\frac{1}{q}}\left\|  u_{q}-\overline{u}\right\|
_{\infty}+\left| \left\|  \overline{u}\right\|  _{q}-\left\|
\overline{u}\right\| _{1}\right|
\end{align*}
and, since $\lim_{q\rightarrow1}\left\|  \overline{u}\right\|
_{q}=\left\|
\overline{u}\right\|  _{1}$, we get%
\[
\lim_{q\rightarrow1}\left\|  u_{q}\right\|  _{q}=\left\| \overline
{u}\right\|  _{1}\,.
\]
Combining these facts we have the claim, namely%
\[
\alpha(p,1,2)\leq\frac{\left\|  \bar{u}^{\prime}\right\|
_{p}}{\left\| \bar{u}\right\|
_{1}}\leq\underset{q\rightarrow1}{\lim\inf}\frac{\left\|
u_{q}^{\prime}\right\|  _{p}}{\left\|  u_{q}\right\|  _{q}}=\lim
_{q\rightarrow1}\alpha(p,q,2)=\bar{\alpha}.
\]

\textit{Step 4.} We now compute $\alpha(\infty,q,r)$. We let%
\[
\mathcal{X}_{r}=\left\{  v\in W_{per}^{1,p}(-1,1),\;\int_{-1}^{1}%
|v|^{r-2}v=0\right\}  .
\]
As above we have%
\[
2^{\frac{1}{p}}\left\|  v^{\prime}\right\|  _{\infty}\geq\left\|
v^{\prime }\right\|  _{p}\geq\alpha(p,q,r)\left\|  v\right\|
_{q}\;,\;\forall
v\in\mathcal{X}_{r}%
\]
and hence, if we denote by%
\[
\bar{\alpha}=\lim\limits_{p\rightarrow\infty}\alpha(p,q,r)=\lim
\limits_{p\rightarrow\infty}\alpha(p,q,q)=\frac{2}{q}\left(  \frac{q+1}%
{2}\right)  ^{\frac{1}{q}}B\left(  1,\frac{1}{q}\right)  =2^{\frac
{1}{q^{\prime}}}\,\left(  q+1\right)  ^{\frac{1}{q}},
\]
we obtain%
\[
\left\|  v^{\prime}\right\|  _{\infty}\geq\bar{\alpha}\left\|
v\right\| _{q}\;,\;\forall v\in\mathcal{X}_{r}.
\]
We therefore proved that
$\alpha(\infty,q,r)\geq\overline{\alpha}$. Let us now show the
reverse inequality. For $p>1$, let $u_{p}\in\mathcal{X}_{r}$ be a
minimizer, i.e.%
\[
\frac{\left\|  u_{p}^{\prime}\right\|  _{p}}{\left\|  u_{p}\right\|  _{q}%
}=\alpha(p,q,r)\,.
\]
Recall that since we assumed $-1\leq-m\leq u_{p}\left(  x\right)
\leq1$ we have $\left\|  u_{p}\right\|  _{\infty}\leq1$. Choosing
$p$\ sufficiently large so that $rp+r-1\geq q$, we can rewrite
(\ref{E-L integree avec mu et c (equivalent)}) as (recalling that
we are then in the case where $m=1$\ and
thus $r\left(  m\right)  =0$)%
\[
|u_{p}^{\prime}|=\alpha(p,q,r)\left\|  u_{p}\right\|  _{q}\left(
\frac{p^{\prime}}{q}\right)  ^{\frac{1}{p}}\left\|  u_{p}\right\|
_{q}^{-\frac{q}{p}}[1-|u_{p}|^{q}]^{\frac{1}{p}},
\]
which implies%
\[
\frac{\left\|  u_{p}^{\prime}\right\|  _{\infty}}{\left\|
u_{p}\right\| _{q}}=\alpha(p,q,r)\left(
\frac{p^{\prime}}{q}\right)  ^{\frac{1}{p}}\left\| u_{p}\right\|
_{q}^{-\frac{q}{p}}\,.
\]
Note that by (\ref{normau}) we have%
\[
\left\|  u_{p}\right\|
_{q}\underset{p\rightarrow\infty}{\longrightarrow }\left(
\frac{2}{q+1}\right)  ^{\frac{1}{q}}\,.
\]
By definition of $\alpha(\infty,q,r)$ we therefore get
\[
\alpha(\infty,q,r)\leq\alpha(p,q,r)\left(
\frac{p^{\prime}}{q}\right) ^{\frac{1}{p}}\left\|  u_{p}\right\|
_{q}^{-\frac{q}{p}}\underset
{p\rightarrow\infty}{\longrightarrow}\bar{\alpha}%
\]
which is the desired inequality.
\end{proof}

\bigskip

\noindent {\bf Acknowledgement. } We thank B. Kawohl for helpful
discussions.

 \medskip
Received August 2002; revised January 2003.

\medskip

\end{document}